\documentclass[12pt]{amsart}
\usepackage[latin1]{inputenc}
\usepackage{mathptmx}
\newtheorem{theorem}{Theorem}[section]
\newtheorem{lemma}[theorem]{Lemma}
\newtheorem{corollary}[theorem]{Corollary}
\newtheorem{proposition}[theorem]{Proposition}

\theoremstyle{definition}
\newtheorem{remark}[theorem]{Remark}

\newtheorem{definition}[theorem]{Definition}
\newtheorem{example}[theorem]{Example}
\newtheorem{remark/example}[theorem]{Remark/Example}

\newtheorem{question}[theorem]{Question}

\let\oldlabel=\label
\def\prellabel{\marginparsep=1em\marginparwidth=44pt
 \def\label##1{\oldlabel{##1}\ifmmode\else\ifinner\else
 \marginpar{{\footnotesize\ \\ \tt
 ##1}}\fi\fi}}

\numberwithin{equation}{section}

\def\PP{ {\bf P} }
\def\NN{ {\bf N} }
\def\ZZ{ {\bf Z} }
\def\QQ{ {\bf Q} }
\def\RR{ {\bf R} }
\def\CC{ {\bf C} }
\def\KK{ {\bf K} }

\def\F{\mathcal F}
\def\FF{\mathbf  F}
\def\MM{\mathbf  M}
\def\GG{\mathbf  G}

\def\PV{\mathbf{PV}}

\newcommand{\Rees}{\operatorname{Rees}}
\newcommand{\ini}{\operatorname{in}}
\newcommand{\gin}{\operatorname{gin}}

\newcommand{\GL}{\operatorname{GL}}

\newcommand{\mm}{\operatorname{{\mathbf m}}}

\newcommand{\depth}{\operatorname{depth}}

\newcommand{\Tor}{\operatorname{Tor}}

\newcommand{\Image}{\operatorname{Image}}

\newcommand{\projdim}{\operatorname{pd}}
\newcommand{\reg}{\operatorname{reg}}

\newcommand{\Sym}{\operatorname{Sym}}

\newcommand{\Ann}{\operatorname{Ann}}
\newcommand{\HF}{\operatorname{HF}}

\newcommand{\Rate}{\operatorname{Rate}}

\numberwithin{equation}{section}

\begin{document}
\centerline{CIME Course 2013, Combinatorial Algebraic Geometry } 
\centerline{Levico Terme - June 10 - June 15, 2013} 
 \bigskip 
 
\title{Koszul algebras and their syzygies}
\author{Aldo Conca  }
\address{ Dipartimento di Matematica,
Universit\`a degli Studi di Genova, Italy} \email{conca@dima.unige.it}
 
\subjclass[2000]{}
\keywords{}
\date{}
\maketitle

\section*{Introduction}
Koszul algebras, introduced  by Priddy in \cite{P},  are positively graded $K$-algebras $R$ whose residue field $K$ has a linear free resolution as an $R$-module. Here linear means that the non-zero entries of the matrices describing the maps in the $R$-free resolution of $K$ have degree $1$. For example, if $S=K[x_1,\dots, x_n]$ is the polynomial ring   over a field $K$ then  $K$ is resolved by the Koszul complex which is linear.      In these lectures we deal  with standard graded commutative $K$-algebras, that is, quotient rings of the polynomial ring $S$ by homogeneous ideals. The program of the lectures is the following:   \medskip 
 
 \noindent {\bf Lecture 1}: Koszul algebras  and Castelnuovo-Mumford regularity. \medskip 
 
 \noindent {\bf Lecture 2}: Bounds for the degrees of the syzygies of Koszul algebras.  \medskip 
 
 \noindent {\bf Lecture 3}:   Veronese algebras and algebras associated with collections of hyperspaces.  \medskip 
 
  In the first lecture, based on the survey paper  \cite{CDR}, we present various characterizations of Koszul algebras and strong versions of Koszulness.   In the second lecture we describe recent results, obtained in cooperation with Avramov and Iyengar \cite{ACI1,ACI2},  on the bounds of the degrees of the syzygies of a Koszul algebra. Finally, the third lecture is devoted to the study of the Koszul property of Veronese algebras and of algebras associated with collections of hyperspaces and it is based on the papers  \cite{CDR,C}.

\section{Koszul algebras and Castelnuovo-Mumford regularity} 

 \subsection{Notation}
\label{s1}
Let $K$ be a field and $R$ be a commutative standard graded $K$-algebra, that is a $K$-algebra with a decomposition  $R=\bigoplus_{i\in \NN} R_i$ (as an Abelian group) such that $R_0=K$,  the vector space $R_1$ has finite dimension and $R_iR_j=R_{i+j}$ for every $i,j\in \NN$.  Let   $S$ be  the symmetric algebra  of $R_1$ over $K$.  In other words, $S$ is the polynomial ring $K[x_1,\dots,x_n]$ where $n=\dim R_1$ and $x_1,\dots,x_n$ is a $K$-basis of $R_1$. One has an induced surjection 
\begin{equation}
\label{canpre}
S=K[x_1,\dots,x_n] \to R
\end{equation} 
of standard graded $K$-algebras. We call (\ref{canpre}) the canonical presentation of $R$. Hence $R$ is isomorphic   to $S/I$ where  $I$ is the kernel of the map (\ref{canpre}). In particular, $I$ is homogeneous and does not contain elements of degree $1$. We  say that $I$ defines $R$.   Denote by $\mm_R$ the maximal homogeneous ideal of $R$. We may consider $K$ as a graded  $R$-module via the identification  $K=R/\mm_R$. \medskip 

Unless otherwise stated, we will always assume that  $K$-algebras are standard graded,  modules and ideals are graded and finitely generated,  and  module homomorphisms have degree $0$. 

\medskip 
  
For an $R$-module $M=\oplus_{i\in \ZZ}  M_i$ we denote by $\HF(M,i)$ the Hilbert function of $M$ at $i$, that is 
$$\HF(M,i)=\dim_K(M_i),$$ and by 
$$H_M(z)=\sum_{i\in \ZZ}  \dim_K(M_i) z^i\in \QQ[|z|][z^{-1}]$$
 the associated Hilbert series. 
 
 Given an integer $a\in\ZZ$ we will denote  by $M(a)$  the graded $R$-module whose degree $i$ component is $M_{i+a}$. In particular $R(-j)$ is a free $R$-module of rank $1$ whose generator has degree $j$. 

A minimal graded free resolution of $M$ as an $R$-module is a complex of free $R$-modules 
$$\FF: \cdots \to  F_{i+1} \stackrel{\phi_{i+1}} \longrightarrow  F_{i}  \stackrel{\phi_{i}} \longrightarrow  F_{i-1} \to \cdots \to F_1\stackrel{\phi_{1}} \longrightarrow F_0\to 0$$
such that: 
\begin{itemize} 
\item[(1)]  $H_i(\FF)=0$ for $i>0$,
\item[(2)] $H_0(\FF)\simeq M$, 
\item[(3)] $\phi_{i+1}(F_{i+1}) \subseteq \mm_R F_i$ for every $i$. 
\end{itemize} 

Such a resolution exists and it is unique up to an isomorphism of complexes. We hence call it  ``the" minimal free (graded) resolution of $M$.  
 
 By definition, the $i$-th Betti number $\beta_i^R(M)$ of $M$ as an $R$-module is the rank of $F_i$. Each $F_i$ is a direct sum of shifted copies of $R$. The $(i,j)$-th graded Betti number $\beta_{ij}^R(M)$ of $M$ is the number of copies of $R(-j)$ that appear in $F_i$.  By construction one has 
 $$\beta_{i}^R(M)=\dim_K  \Tor^R_i(M,K)$$  and $$\beta_{ij}^R(M)=\dim_K  \Tor^R_i(M,K)_j.$$
 
 Here and throughout the notes an index on the right of a graded module denotes the homogeneous component of that degree. 
 
  The Poincar\'e series of $M$ is defined as 
$$P_M^R(z)=\sum_{i\in \NN}  \beta_i^R(M)z^i\in \QQ[|z|],$$ 
and its bigraded version is 
$$P_M^R(s,z)=\sum_{i\in \NN,j\in \ZZ}  \beta_{i,j}^R(M)s^jz^i\in \QQ[s,s^{-1}][|z|].$$ 
We set 
$$t_i^R(M)=\sup\{ j : \beta_{ij}^R(M)\neq 0\}$$ 
where, by convention,  $t_i^R(M)=-\infty$ if $F_i=0$. By definition, $t_0^R(M)$ is the largest degree of a minimal generator of $M$. 
Two important invariants  that measure the ``growth" of the resolution of $M$ as an $R$-module are the projective dimension  
$$\projdim_R(M)=\sup\{ i : F_i\neq 0\}=\sup\{ i : \beta_{i}^R(M)\neq 0\}$$
and the (relative) Castelnuovo-Mumford regularity 
$$\reg_R(M)=\sup\{ j-i : \beta_{ij}^R(M)\neq 0\}=\sup\{ t_i^R(M)-i : i\in \NN\}.$$

An $R$-module $M$ has a linear resolution as an $R$-module  if for some $d\in \ZZ$ one has $\beta_{ij}^R(M)=0$ if $j\neq d+i$. Equivalently,  $M$ has a  linear resolution as an $R$-module if it  is generated by elements of degree $\reg_R(M)$. 
 
We may as well consider $M$ as a module over the polynomial ring $S$ via the map (\ref{canpre}).  The absolute Castelnuovo-Mumford regularity  is, by definition, the regularity $\reg_S(M)$ of $M$ as a $S$-module. It has also a cohomological interpretation via local duality, see  for example 
\cite[Section 1]{EG} or \cite[4.3.1]{BH}. Denote  by $H^i_{\mm_S}(M)$ the $i$-th local cohomology module with support on the maximal ideal of $S$. One has $H^i_{\mm_S}(M)=0$ if $i<\depth M$ or  $i>\dim M$ and 
 $$\reg_S(M)=\max\{ j+i : H^i_{\mm_S}(M)_j\neq 0\}.$$
 
  Both $\projdim_R(M)$ and $\reg_R(M)$ can be infinite. 
 \begin{example}  
 Let  $R=K[x]/(x^{v})$ with $v>1$.  Then the minimal free resolution of $K$ over $R$ is: 
 $$\cdots \to R(-2v)\to   R(-v-1)\to R(-v)\to  R(-1)\to R\to 0$$
 where the maps are multiplication by $x$ or $x^{v-1}$ depending on the parity. 
 Hence  $F_{2i}=R(-iv)$ and $F_{2i+1}=R(-iv-1)$ so that  $\projdim_R(K)=\infty$ for every $v>1$. Furthermore   $\reg_R(K)=\infty$ if $v>2$ and $\reg_R(K)=0$ if $v=2$. 
 \end{example} 
  
 Note that, in general, $\reg_R(M)$ is finite if $\projdim_R(M)$ is finite. On the other hand, as we have seen in the example above,  t  $\reg_R(M)$ can be  finite even  when  $\projdim_R(M)$ is infinite.

In the study of   minimal free resolutions over $R$,   the minimal free resolution $\KK_R$ of  the residue field $K$ as an $R$-module plays a fundamental role. This is because   $$\Tor_*^R(M,K)=H_*(M\otimes_R \KK_R)$$ and hence 
$$\beta_{ij}^R(M)=\dim_K H_i(M\otimes_R \KK_R)_j.$$ 

A very important role is played also by the Koszul complex $K(\mm_R)$ on a minimal system of generators of the maximal ideal $\mm_R$ of $R$. 
The Koszul complex is the typical example of a  differential graded algebra, DG-algebra for short.  
 
 \subsection{DG-algebras}
 A graded algebra 
 $$C=\oplus_{i\geq 0} C_i$$ 
 is  graded-commutative if for every $a\in C_i$ and $b\in C_j$ one has: 
$$ab=(-1)^{ij}ba$$
and furthermore 
$$a^2=0$$
whenever $i$ is odd. 

 A DG-algebra is a graded-commutative algebra  $C=\oplus_{i\geq 0} C_i$ equipped with a linear differential $$\partial:C\to C$$ of degree $-1$ (i.e. $\partial^2=0$ and $\partial(C_i)\subseteq C_{i-1}$)  that  satisfies the ``twisted" Leibniz rule: 
$$\partial(ab)=\partial(a)b+(-1)^i a\partial(b)$$
whenever $a\in C_i$. 

The cycles $Z(C)=\ker \partial$, the boundaries $B(C)=\Image \partial$  and the homology $H(C)=Z(C)/B(C)$ of a DG-algebra $C$  inherit the algebra structure from $C$. Precisely, $Z(C)$ is a graded-commutative subalgebra of $C$, $B(C)$ is a (two-sided) graded ideal of $Z(C)$ and hence $H(C)$ is a (graded-commutative) algebra.  

The component of  $Z(C)$ of (homological) degree $i$  is denoted by  $Z_i(C)$.   Similarly for the boundaries and the homology. 

Given a DG-algebra $C$ and a cycle $z\in Z_i(C)$ there is a canonical way to ``kill" $z$ in homology  by adding a ``variable" to $C$ preserving the  DG-algebra structure. If $i$ is even then  one considers $D=C[e]$ where $e$ is an exterior variable (hence $e^2=0$) of degree $i+1$ and extends the differential by setting $\partial(e)=z$. If $i$ is odd then one considers $D=C[s]$ where $s$ is a polynomial variable (or a divided power variable) of degree $i+1$ and extends the differential by setting $\partial(s)=z$. By construction, the element $z\in Z(C)\subset Z(D)$ is now a boundary of the complex $D$ and hence it is $0$ in homology. Furthermore, by construction,  $H_j(C)=H_j(D)$ for $j<i$. This process can clearly be iterated. One can, for instance, ``kill" all the cycles in a given homological degree by adding variables.  

\subsection{Koszul complex} 
The Koszul complex can be described in the following way. Let $R$ be any ring and let $I=(a_1,\dots,a_m)$ be an ideal of $R$. Consider $R$ as a DG-algebra concentrated in degree $0$ and the elements $a_1,\dots,a_m$  as cycles of that complex. Then we add exterior variables $e_1,\dots,e_m$ in degree $1$ to $R$ and obtain the DG-algebra 
$$K(I,R)=R[e_1,\dots,e_m] \mbox{ with } \partial(e_i)=a_i.$$
This is the Koszul complex associated with the ideal $I$ and coefficients in the ring $R$. In other words, $K(I,R)$ is the exterior algebra $\bigwedge R^m$ equipped with  the differential induced by $\partial(e_i)=a_i$ for $i=1,\dots,m$. 
If $M$ is an $R$-module we then set 
$$K(I,M)=K(I,R)\otimes_R M.$$
This is the Koszul complex associated with the ideal $I$ with coefficients in  $M$. 
Denote by  $Z(I,M)$ the module of  cycles of the complex $K(I,M)$ and similarly by $B(I,M)$ its   boundaries, by $C(I,M)$ its cokernel and  by  $H(I,M)$ its homology. We will denote by $K_i(I,M)$ the component of homological  degree $i$ of $K(I,M)$.  
When the coefficients of the Koszul complex are taken in $R$ we use a simplified notation 
$$K(I)=K(I,R), \quad Z(I)=Z(I,R)$$ 
and so on.   

By definition we have: 
$$H_0(I)=R/I \mbox{ and } H_0(I,M)=M/IM.$$
Furthermore  $H(I)$ is a (graded-commutative) algebra and $H(I,M)$ is a $H(I)$-module. In particular, $IH(I,M)=0$. 

It is well-known that the Koszul complex  $K(I)$ is acyclic (and hence an $R$-free resolution of $R/I$)  if (and only if in the local  or standard graded setting) the chosen generators $a_1,\dots,a_m$  of $I$  form a regular sequence, see \cite[1.6.14]{BH}.

When $K(I)$ is not a free resolution one can nevertheless use the procedure of  adding variables to kill homology degree by degree  to obtain from $K(I)$ a free resolution of $R/I$ as an $R$-module with a DG-algebra structure. In the local or standard graded  setting this can be done in the following way. 
\begin{itemize} 
\item[(1)] Set $T_1=K(I)$.  

\item[(2)]  Choose a minimal system of generators of $H_1(T_1)$ and a set of cycles $z_{1,1},\dots,z_{1,u_1}$ representing them. 

\item[(3)]  Add to $T_1$ a set of polynomial variables  (or divided powers in positive characteristic)  $Y_2=\{s_{2,1},\dots,s_{2,u_1}\}$ in degree $2$  and set  $$T_2=T_1[s_{2,1},\dots,s_{2,u_1} : \partial(s_{2,i})=z_{1,i}].$$

\item[(4)]  Choose a minimal system of generators of $H_2(T_2)$, a set of cycles   $z_{2,1},\dots,z_{2,u_2}$ representing them.

\item[(5)]  Add to $T_2$ a set of  exterior  variables    $Y_3=\{e_{3,1},\dots,e_{3,u_2}\}$ in degree $3$  and set 
$$T_3=T_2[e_{3,1},\dots,e_{3,u_2} : \partial(e_{3,i})=z_{2,i}]$$
\end{itemize}  
 and so on. We obtain a DG-algebra $T=R[Y_1,Y_2,Y_3,\dots]$ that is an $R$-free resolution of $R/I$ which is (essentially) independent on the choices of   the  minimal system of generators of $H_i(T_i)$ and of the cycles representing them. It is called the Tate complex and we will denote it by 
$$T(R,R/I).$$ 
We refer to \cite{A} for a precise description of the construction and many beautiful results and questions concerning it. We just give two examples:   

\begin{example}
\label{Tate1} Set  $R=\QQ[x]/(x^n)$, $a=\bar{x}$ and $I=(a)$. Then $T_1=K(I)=R[e]$ with $\partial(e)=a$ and $z=a^{n-1}e$ generates the $1$-cycles of  $K(I)$. Hence the second  round of Tate construction gives $T_2=R[e,s]$ with $\partial(s)=a^{n-1}e$. It turns out that $T(R,R/I)=R[e,s]$ is a minimal resolution of $R/I$ over $R$: 
$$\cdots \to Rs^2e \to Rs^2 \to Rse\to Rs\to Re\to R\to 0$$
where the maps are given, alternatively,  by multiplication with $a$ and $a^{n-1}$ up to non-zero scalars. 
\end{example} 

\begin{example} 
\label{Tate2} 
Set  $R=\QQ[x,y]$ and $I=(x^2,xy)$. Then we have  $T_1=K(I)=R[e_1,e_2]$ with $\partial(e_1)=x^2$ and  $\partial(e_2)=xy$.  The cycle $ye_1-xe_2$  generates $H_1(T_1)$. Hence the second  round of Tate construction gives the DG-algebra $T_2=R[e_1,e_2,s_1]$ with $\partial (s_1)=ye_1-xe_2$. Now $H_2(T_2)$ is generated by $e_1e_2-ys_1$. Hence the third  round of Tate construction gives $T_3=R[e_1,e_2,s_1,e_3]$ with $\partial(e_3)=e_1e_2-ys_1$. Hence the beginning of the Tate complex is the following: 
$$\cdots\to Re_3\oplus Rs_1e_1\oplus Rs_1e_2\to Rs_1 \oplus Re_1e_2 \to Re_2\oplus Re_1\to R\to 0.$$
\end{example} 
Note that the resolution in Example \ref{Tate1} is minimal while the one in  Example \ref{Tate2} is not.  

 \subsection{Auslander-Buchsbaum-Serre}
We return to the graded setting and assume that  $R$ is a standard graded $K$-algebra. 
When is $\projdim_R(M)$ finite  for every $M$? The answer is given  by the   Auslander-Buchsbaum-Serre Theorem, a result that, in the words of Avramov  \cite[p.32]{A1},  ``really started everything".  The graded variant of it is the following: 
\begin{theorem} 
\label{ABS}
The following conditions are equivalent:
\begin{itemize}
\item [(1)]  $\projdim_R(M)$ is finite for every $R$-module $M$, 
\item[(2)]  $\projdim_R(K)$ is finite, 
\item[(3)] the Koszul complex $K(\mm_R)$ resolves $K$,
\item[(4)]  $R$ is a polynomial ring. 
\end{itemize} 
When the conditions hold, then for every $M$ one has $\projdim_R(M)\leq \projdim_R(K)=\dim R$. 
\end{theorem} 

\begin{remark}
\label{koscomplex} 
The Koszul complex $K(\mm_R)$ has three important features:
\begin{itemize}
\item[(1)]  it has  finite length,
\item[(2)] it has a DG-algebra structure,
\item[(3)]  the matrices describing its differentials have non-zero entries only of degree $1$.
\end{itemize}
 \end{remark}
 
When $R$ is not a polynomial ring  the minimal free  resolution $\KK_R$  of $K$ as an $R$-module does not satisfy  condition (1) of Remark \ref{koscomplex}. Can  $\KK_R$ nevertheless satisfy conditions (2) and  (3) of Remark \ref{koscomplex}? 

For condition (2) the answer is yes: $\KK_R$ has always a DG-algebra structure. Indeed a theorem, proved independently by  Gulliksen and Schoeller,  asserts that $\KK_R$ is obtained by using the Tate construction. Furthermore results of Assmus, Tate, Gulliksen and Halperin clarify when  $\KK_R$ is finitely generated as an $R$-algebra. Again, we state the theorem in the graded setting and refer to  \cite[6.3.5, 7.3.3, 7.3.4]{A} for general statements and proofs.

\begin{theorem}\label{SuperT} Let $R$ be a standard graded $K$-algebra. Let $T=T(R,K)=R[Y_1,Y_2,Y_3,\dots]$ be the Tate complex associated with $K=R/\mm_R$. We have:
\begin{itemize} 
\item[(1)]  $T$  is the minimal graded resolution of $K$, i.e. $T=\KK_R$. 
\item[(2)] $T$ is finitely generated as an $R$-algebra if and only if $R$ is a complete intersection.  In that case, $T$ is generated by elements of degree at most $2$, i.e. $Y_i=\emptyset$ for $i>2$. 
\item[(3)] If $R$ is not a complete intersection then $Y_i\not = \emptyset$ for every $i$. 
\end{itemize}
\end{theorem} 

Algebras $R$ such that $\KK_R$ satisfies condition (3) in Remark \ref{koscomplex}  are called Koszul: 

\begin{definition} The $K$-algebra $R$ is Koszul if  the matrices describing  the differentials in the minimal free resolution $\KK_R$ of $K$ as an $R$-module have non-zero entries only of degree $1$, that is, $\reg_R(K)=0$ or, equivalently,  $\beta_{ij}^R(K)=0$ whenever $i\neq j$. 
\end{definition}

Koszul algebras were originally introduced  by Priddy \cite{P} in his study of  homological properties of graded (non-commutative) algebras, see the volume   \cite{PP} of  Polishchuk and Positselski for an overview and surprising aspects of the Koszul property. 
We collect below  a list of important facts about Koszul commutative algebras. We always refer to the canonical presentation (\ref{canpre}) of $R$ as a quotient of the polynomial ring $S=\Sym(R_1)$ by the homogeneous ideal $I$. 
First we introduce a definition.

\begin{definition}\label{quad} 
We say that $R$ is quadratic if its defining ideal $I$ is generated by quadrics (homogeneous polynomials of degree $2$). 
 \end{definition}

\begin{definition}\label{Gquad} 
 We say that $R$ is G-quadratic if  its defining ideal $I$ has a Gr\"obner basis of quadrics with respect to some coordinate system of $S_1$ and some term order $\tau$ on $S$. In other words, $R$ is $G$-quadratic if there exists a $K$-basis of $S_1$, say $x_1,\dots,x_n$ and a term order $\prec$  such that the initial ideal $\ini_\prec(I)$ of $I$ with respect to $\prec$ is generated by monomials of degree $2$. 
 \end{definition}

\begin{remark} 
\label{R161}
For a standard graded $K$-algebra $R$ one has 
$$\beta_{2j}^R(K)=\left\{\begin{array}{lr}
\beta_{1j}^S(R) &\mbox{ if } j\neq 2\\ \\
\beta_{12}^S(R)+\binom{n}{2} &\mbox{ if } j=2
\end{array}
\right.
 $$
and hence the resolution of $K$, as an $R$-module, is linear up to homological position $2$ if and only if $R$ is quadratic. 
In particular, if $R$ is Koszul, then $R$ is quadratic.

On the other hand there are algebras defined by quadrics that are not Koszul.  For example if one takes 
$$R=K[x,y,z,t]/(x^2,y^2,z^2,t^2,xy+zt)$$
 then one has $\beta_{34}^R(K)=5$ and hence $R$ is not Koszul.   
\end{remark} 

\begin{remark} 
\label{R162}
If $I$ is generated by monomials of degree $2$ with respect to some coordinate system of $S_1$, then a filtration argument,  that we reproduce in Theorem \ref{monKos}, shows  that $R$ is Koszul. More precisely,  for every subset $Y$ of variables $R/(Y)$ has an $R$-linear resolution.  \end{remark} 

\begin{remark} 
\label{R163}
If $I$ is generated by a regular sequence of quadrics, then $R$ is Koszul. This follows from  Theorem \ref{SuperT} because if $R$ is a complete intersection of quadrics, then $\KK_R$ is obtained from $K(\mm_R)$ by adding polynomial variables in homological degree $2$ and internal degree $2$  to kill $H_1(K(\mm_R))$. 
For example, if 
$$R=\QQ[x_1,x_2,x_3,x_4]/(x_1^2+x_2^2, x_3x_4)$$ then the  Tate resolution of $K$ over $R$ is the DG-algebra
 $$R[e_1,e_2,e_3,e_4,s_1,s_2]$$ with differential induced by $\partial(e_i)=x_i$ and $\partial(s_1)=x_1e_1+x_2e_2$ and $\partial(s_2)=x_3e_4$. Here the $e_i$'s  have internal degree $1$ and the $s_i$'s have internal degree $2$.  
\end{remark}

\begin{remark} 
\label{R164}
If $R$ is G-quadratic,  then $R$ is Koszul. This follows from  Remark \ref{R162} and from  the standard deformation argument showing that $\beta_{ij}^R(K)\leq \beta_{ij}^{A}(K)$ with $A=S/\ini_\tau(I)$. For details see, for instance, \cite[Section 3]{BC}.  
\end{remark} 

\begin{remark} 
\label{R165}
On the other hand there are Koszul algebras that are not G-quadratic. One notes that  an ideal defining a G-quadratic  algebra must contain quadrics of ``low" rank. For instance, if $R$ is Artinian and G-quadratic then its defining ideal must contain the square of a linear form. But most Artinian complete intersections of quadrics do not contain the square of a linear form. For example, the ideal $I=(x^2+yz, y^2+xz, z^2+xy)\subset S=\CC[x,y,z]$ is a complete intersection not containing the square of a linear form and $S/I$ is Artinian. Hence $S/I$ is   Koszul and not G-quadratic. See \cite{ERT} for general results in this direction. 
\end{remark} 

\begin{remark} 
\label{R166}
The Poincar\'e series $P_K^R(z)$ of $K$ as an $R$-module can be irrational (even for rings with $R_3=0$), see \cite{An}. However for a Koszul algebra $R$ one has 
\begin{equation} 
P_K^R(z)H_R(-z)=1 
\label{HilPoi} 
\end{equation}
and hence $P_K^R(z)$ is rational. Indeed the equality (\ref{HilPoi}) turns out to be equivalent to the Koszul property of $R$, see for instance \cite{F}.
\end{remark} 

\begin{remark} 
\label{R167}
 A necessary (but not sufficient) numerical condition for $R$ to be Koszul is that the formal power series $1/H_R(-z)$ has non-negative coefficients (indeed positive unless $R$ is a polynomial ring). 
 Another numerical condition is the following.  Expand the formal power series $1/H_R(-z)$ as 
$$\frac{ \Pi_{h \in 2\NN+1}  (1+z^h)^{e'_h}}{ \Pi_{h\in 2\NN+2} (1-z^h)^{e'_h}}$$
with $e'_h\in \ZZ$. This can be done in a unique way,  see \cite[7.1.1]{A}. 
Furthermore set $e_h(R)=\# Y_h$ where $Y_h$ is the set of variables that we add at the  $h$-th iteration of the Tate construction of the minimal free resolution of $K$ over $R$.  The numbers $e_h(R)$ are called the deviations of $R$. 
 If $R$ is Koszul then $e'_h=e_h(R)$ for every $h$ and hence $e_h'\geq 0$. More precisely, $e_h'>0$ for every $h$ unless $R$ is a complete intersection. 
   
For example,  the Hilbert function  of the ring in Remark \ref{R161}  is  
$$H(z)=1+4z+5z^2.$$  Expanding the series $1/H(-z)$ one sees that the coefficient of $z^6$ is negative.  Furthermore the corresponding $e'_3$ is $0$. So for two numerical reasons an algebra with Hilbert series $H(z)$ cannot be Koszul.    
\end{remark} 

The following  characterization of the Koszul property in terms of  regularity  is formally similar to the Auslander-Buchsbaum-Serre Theorem \ref{ABS}.

\begin{theorem}[Avramov-Eisenbud-Peeva]
\label{AEP}
The following conditions are equivalent:
\begin{itemize}
\item[(1)]  $\reg_R (M)$ is finite for every $R$-module $M$, 
\item[(2)]  $\reg_R(K)$ is finite, 
\item[(3)] $R$ is Koszul. 
\end{itemize} 
\end{theorem}  

 Avramov and Eisenbud proved in \cite{AE} that every module $M$ has finite regularity over a Koszul algebra $R$ by showing that $\reg_R(M)\leq \reg_S(M)$.  Avramov and Peeva showed in \cite{AP} that if  $\reg_R (K)$ is finite then it must be $0$. Indeed they proved a more general result for graded algebras that are not necessarily standard.

 \medskip
 
We collect below  further remarks, examples and questions  relating the Koszul property and the existence of Gr\"obner bases of quadrics in various ways. Let us recall the following:
 
 \begin{definition}
 \label{LG-quad}
 A $K$-algebra $R$ is LG-quadratic (where the L stands for lifting) if there exist a G-quadratic algebra $A$ and a regular sequence of linear forms $y_1,\dots,y_c$ such that $R\simeq A/(y_1,\dots,y_c)$. 
 \end{definition}

 We have the following implications: 
 \begin{equation}
 \label{GLG}
   \mbox{G-quadratic } \Rightarrow \mbox{LG-quadratic }  \Rightarrow \mbox{Koszul}  \Rightarrow \mbox{quadratic }
   \end{equation}
 
As discussed in Remark \ref{R161} and Remark \ref{R167} the third  implication in  (\ref{GLG}) is  strict.  The following remark, due to Caviglia, in connection with Remark \ref{R165} shows that also the first implication in  (\ref{GLG}) is strict. 
 \begin{remark} 
 Any complete intersection $R$ of quadrics is LG-quadratic. 
 Say 
 $$R=K[x_1,\dots,x_n]/(q_1,\dots,q_m)$$ 
 is a complete intersection of quadrics. Then set 
 $$A=R[y_1,\dots,y_m]/(y_1^2+q_1,\dots,y_m^2+q_m)$$ and note that $A$ is G-quadratic because  the initial ideal of its defining ideal with respect to a  lex term order satisfying  $y_i>x_j$ for every $i,j$ is $(y_1^2,\dots,y_m^2)$. Furthermore  $y_1,\dots,y_m$ is a regular sequence in $A$ because 
 $$A/(y_1,\dots,y_m)\simeq R$$ and $\dim A-\dim R=m$.
 \end{remark} 
 
  In   \cite[1.2.6]{Ca1}, \cite[6.4]{ACI1} and \cite[12]{CDR} it is asked whether a Koszul algebra is also LG-quadratic. The following example gives a negative answer to the question. 
  
 \begin{example}
\label{KosNotLG} 
 Let $$R=K[a,b,c,d]/(ac, ad, ab - bd, a^2 + bc, b^2).$$
 The Hilbert series of $R$ is $$\frac{1+2z-2z^2-2z^3+2z^4}{(1-z)^2}.$$  Also, $R$ is Koszul as can be shown using a Koszul  filtration argument, see Example \ref{non-LG with filtration}.  The $h$-polynomial does not change under lifting with regular sequences of linear forms.   Hence to check that $R$  is not  LG-quadratic  it is enough to check that  there is no algebra with  quadratic monomial relations and  with $h$-polynomial equal to
 $$h(z)=1+2z-2z^2-2z^3+2z^4.$$ 
 
 In general, if  $J$ is an ideal in  a polynomial ring $A$ not containing linear forms and the  $h$-polynomial of $A/J$ is $1+h_1z+h_2z^2+\dots$ then $J$ has codimension $h_1$ and exactly $\binom{h_1+1}{2}-h_2$ quadratic generators. 

 Now consider a quadratic monomial ideal $J$ in a polynomial ring $A$ with,  say, $n$ variables  such that  the  $h$-polynomial of $A/J$ is $h(z)$. 

Such a $J$ must  have codimension $2$ and  $5$ generators. So $J$ is generated by $5$ monomials chosen among the generators of $(a,b)(a,b,c,d,e,f,g)$ where $a,b,\dots,g$ are distinct variables.  But an exhaustive CoCoA \cite{CoCoA} computation  shows that such a selection does not exist.
\end{example} 

 An interesting example of LG-quadratic algebra is the following: 

\begin{example}
\label{LGNonObstructed} 
Let 
$$R=K[a,b,c,d]/(a^2-bc, d^2, cd, b^2, ac, ab).$$
The Hilbert series of $R$ is $$\frac{1+3z-3z^3}{(1-z)}.$$  The $h$-polynomial  $1+3z-3z^3$ is not the $h$-polynomial of a quadratic  monomial ideal  in $4$ variables. It is however 
the $h$-polynomial of a (unique up to permutations of the variables)  quadratic  monomial ideal in $5$ variables, namely $U_1=(a^2, b^2, ad, cd, be, ce)$. 
In $6$ variables there is another quadratic monomial ideal with that $h$-polynomial. It is $U_2=(a^2, ad, bd, be, ce, cf)$.  Another one  in $7$ variables, $U_3=(ad, bd, ae, ce, bf, cg)$.  And that is all. 
It turns out that $R$ is LG-quadratic since it lifts to $K[a,b,c,d,e]/J$ with $$J=(a^2 - bc + be, d^2, cd, b^2+eb, ac, ab + ae)$$
 and $\ini_\prec(J)$ is $U_1$ (up to a permutation of the variables) where  $\prec$  is the rev.lex. order associated with the total order of the variables $e>d>b>c>a$. 
\end{example}

The ring of Example \ref{KosNotLG} is not LG-quadratic because of the obstruction in the $h$-polynomial, i.e., there are no quadratic monomial ideals with that $h$-polynomial.  It would be interesting to identify a Koszul algebra with a  non-obstructed $h$-polynomial that is not LG-quadratic.

 \section{Syzygies of Koszul algebras} 
 The second lecture is based on the results obtained jointly with  Avramov and Iyengar and   published in   \cite{ACI1,ACI2}. 
 
 Given a standard graded $K$-algebra $R$ with presentation $$R=S/I$$ where $S=K[x_1,\dots,x_n]$ and $I\subset S$ is an homogeneous ideal, we set 
 $$t_i^S(R)=\sup\{ j : \beta_{ij}^S(R)\neq 0\}$$
 i.e., $t_i^S(R)$ is the highest degree of a minimal $i$-th syzygy of $R$ as a $S$-module. In particular, $t_0^S(R)=0$ and $t_1^S(R)$ equals the 
   highest degree of a minimal generator of $I$. 

The starting point of the discussion is the following observation. 

\begin{remark} 
\label{boundsyz}
If  $I$ is generated by monomials of degree $2$, then:
 \begin{itemize} 
 \item[(1)] $t_i^S(R)\leq 2i$ for every $i$,
 \item[(2)] $\reg_S(R)\leq \projdim_S(R)$,
 \item[(3)] if $t_i^S(R)<2i$ for some $i$  then  $t_{i+1}^S(R)<2(i+1)$,
 \item[(4)] $t_i^S(R)<2i$ if $i>\dim S-\dim R$, 
 \item[(5)] $\beta^S_i(R)\leq \binom{\beta_1^S(R)}{i}$,
 \item[(6)] $\projdim_S(R)\leq \beta_1^S(R)$.
 \end{itemize} 
These inequalities  are deduced from  the  (non-minimal) Taylor resolution of quadratic monomial ideals, see for instance \cite[4.3.2]{MS}.
\end{remark}

Suppose now that, combining and iterating the following operations: 
\begin{itemize} 
\item[(i)] changes of coordinates, 
\item[(ii)] the formation of  initial ideals with respect to weights or term orders,
\item[(iii)] liftings and specializations with regular sequences of degree $1$, 
\end{itemize} 

the algebra $R$ deforms  to  an algebra $R'=S'/J$ with $S'$ a polynomial ring and $J$ generated by  quadratic monomials.  Then $R$ satisfies the inequalities (1), (2), (5) and (6) because the Betti numbers and the $t_i$'s can only grow passing from $R$ to $R'$ and $\beta_{1}^S(R)=\beta_{12}^S(R)$ equals $\beta_{1}^{S'}(R')=\beta_{12}^{S'}(R')$.  This observation suggests the following question: 

\begin{question}
Are the bounds of Remark \ref{boundsyz} valid for every Koszul algebra? 
\end{question}

In  \cite{ACI1} we have proved that (1), (2), (3) and (4)   hold for every Koszul algebra.  As far as we know it is still  open whether (5) and (6) hold as well for Koszul algebras.  The inequality (1) for Koszul algebras (and its immediate consequence (2)) has a short proof that we present  below, see Lemma \ref{soeasy}. 

In \cite{ACI2} stronger limitations for the degrees of the syzygies of Koszul algebras are described (under mild assumptions 
on the characteristic of the base field).  To explain the results in \cite{ACI1} and \cite{ACI2} we start from some  general considerations concerning  bounds on the $t_i$'s. 

For $S=K[x_1,\dots,x_n]$ and $R=S/I$ (and no assumptions on $R$), a very basic question is whether one can bound $t_i^S(R)$ only in terms of $t_1^S(R)$ and the index $i$. The answer is negative in this generality, but it is positive if one involves also the number of variables $n$.  
Indeed, in \cite{BM} and \cite{CaS} it is proved that 
\begin{equation}
 \label{intro0}
t_i^S(R) \leq (2t_1^S(R))^{2^{n-2}}-1+i.  
\end{equation}
Furthermore variations of the Mayr-Meyer  ideals   define algebras having a doubly exponential syzygy growth of the kind of the right-hand side of (\ref{intro0}) (but with slightly different coefficients), see \cite{BS,Ko}.   So,  without any further assumption,    one cannot expect any better bound for the $t_i^S(R)$ in terms of $t_1^S(R)$ than the one  derived from (\ref{intro0}). 

Things change drastically if either $R$ is defined by monomials (i.e. $I$ is a monomial ideal) or $R$ is Koszul.   Under these assumptions we have  that: 
 \begin{equation}
 \label{intro1}
 t_i^S(R)\leq t_1^S(R)i 
 \end{equation}
 holds for every $i$.  
  In particular 
 \begin{equation}
 \label{intro1c}
  t_i^S(R)\leq 2i \ \mbox{ for every $i$}
   \end{equation}
holds for Koszul algebras since $t_1^S(R)=2$. 

 When $R$ is defined by monomials (\ref{intro1}) is  derived from the Taylor resolution, while when $R$ is Koszul 
 (\ref{intro1c}) is  proved by Kempf \cite{K}, in  \cite{ACI1}   and also in an unpublished manuscript of Backelin \cite{B2}. 

 One can ask whether the inequalities (\ref{intro1}) and   (\ref{intro1c})  are  special cases of, or can be derived from,  more  general statements. 
We consider the following  generalization of  (\ref{intro1}): 
\begin{equation}
 \label{intro3}
 t_{i+j}^S(R)\leq t^S_{i}(R)+t^S_{j}(R)   \mbox{ for every } i \mbox{ and } j.
  \end{equation}

No counterexample is known to us to the validity of  (\ref{intro3}) for algebras with monomial relations or for Koszul algebras.

Also,  (\ref{intro3}) for $j=1$  holds for algebras defined by  monomials, see \cite[1.9]{FG} where the statement is proved when  $R$ is defined by monomials of degree $2$ and  \cite{HS} for the general case. 
Furthermore  in \cite[4.1]{EHU} it is proved that  (\ref{intro3}) holds for algebras of Krull dimension at most $1$ when  $i+j=n$, see also \cite{Mc} for related results.

Denote by  
$$Z=\bigoplus_{i\geq 0}  Z_i=\bigoplus_{i\geq 0}  Z_i(\mm_R),$$
$$B=\bigoplus_{i\geq 0}  B_i=\bigoplus_{i\geq 0}  B_i(\mm_R),$$
and by
$$H=\bigoplus_{i\geq 0}  H_i=\bigoplus_{i\geq 0}  H_i(\mm_R)$$  
  the modules of cycles, boundaries and homology  of the Koszul complex $K(\mm_R)$ associated  with the maximal homogeneous ideal $\mm_R$ of $R$. Similarly,  $Z_i(\mm_R,M)$ stands for the $i$-th cycles of $K(\mm_R,M)=K(\mm_R)\otimes M$ and so on. 
  By construction, 

\begin{equation}
 \label{intro3b}
t_i^S(R)=\sup\{ j : (H_i)_j\neq 0\}.
  \end{equation}

For a Koszul algebra $R$ we have shown in \cite{ACI1} that the map\begin{equation}
\wedge^i H_1\to H_i
\end{equation} 
  induced by the multiplicative structure on $H$ is surjective in  degree $2i$  and higher. Hence   \ref{intro1c}  (for Koszul algebras)  is an immediate  corollary of that assertion. 
The inequality  (\ref{intro3}) would follow from a similar statement regarding the multiplication map 
\begin{equation}
 \label{intro4}
H_i\otimes H_{j} \to H_{i+j}.
  \end{equation}
Indeed  it would be enough to prove that the map \ref{intro4} is surjective in degree  $t_i^S(R)+t_j^S(R)$  and higher.  Unfortunately we are not able to evaluate  directly the cokernel of the map \ref{intro4}. Instead we  can get some information  by using the splitting map for Koszul cycles described originally in \cite{BCR2} and rediscussed in  \cite{ACI2} from a more general perspective.  Indeed in   \cite[2.2] {ACI2}  it is proved that:  

\begin{theorem} 
\label{rimpo}
Let $M$ be a graded  $R$-module. For even $i,j$ there is a natural  map (of degree $0$) 
\begin{equation} 
\label{splitmap}
\Tor_1^R(C_{i-1}, Z_j(\mm_R,M))\to H_{i+j}(\mm_R, M)/H_{i}H_{j}(\mm_R, M)
\end{equation} 
that is surjective provided $R$ has characteristic $0$ or large. Here $C_{i-1}$ denotes the cokernel of the Koszul complex $K(\mm_R)$ in homological position $i-1$. 
\end{theorem} 
Taking $M=R$ one obtains  a natural map: 
$$\Tor_1^R(C_{i-1}, Z_j)=\Tor_1^R(B_{i-1},B_{j-1})\to H_{i+j}/H_iH_j$$
that is surjective in characteristic $0$ or large.   Note that $\Tor_1^R(B_{i-1},B_{j-1})$  is a finite length module because the $B_u$'s are free when localized at any non-maximal prime homogeneous ideal.   In particular one has: 

\begin{proposition} Set $T_{ij}=\Tor_1^R(B_{i-1},B_{j-1})$. If $R$ has characteristic $0$ or large then 
\begin{equation}
 \label{intro4b}
t_{i+j}^S(R)\leq \max\{t_{i}^S(R)+t_{j}^S(R),  \reg_S T_{ij} 
  \}  
   \end{equation}
where  
$$ \reg_S T_{ij} =  \max\{ v :  (T_{ij})_v\neq 0  \}.$$
\end{proposition}

In order to evaluate   $\reg_S T_{ij}$   we have developed in \cite{ACI2} a long and technically complicated  inductive procedure. The results obtained take a simpler form in the Cohen-Macaulay case because, under such assumption, the $t_i^S(R)$'s form an increasing sequence. 
We have: 

 \begin{theorem} 
 \label{mainACI1} Let  $R$ be a Koszul and  Cohen-Macaulay algebra of characteristic $0$.  Then 
 \begin{equation}
    \label{intro6b}
      t_{i+1}^S(R)\leq t_i^S(R)+t_1(R)=t_i^S(R)+2
  \end{equation}
 and 
\begin{equation}
    \label{intro6}
      t^S_{i+j}(R)\leq \max\{ t_i^S(R)+t_j^S(R) , t_{i-1}^S(R)+t_{j-1}^S(R)+3\}   
  \end{equation}
hold for every  $i$ and $j$. 
\end{theorem} 

Furthermore one also deduces: 
 \begin{theorem} 
 \label{mainACI2}
If $R$ is a Koszul algebra of characteristic $0$ satisfying the Green-Lazarsfeld $N_p$ condition for some $p>1$ (i.e. $\beta^S_{ij}(R)=0$ for every $i=1,2,\dots,p$ and every $j>i+1$) then  
\begin{equation}
 \label{intro7}
 \reg_S(R)\leq   2\left\lfloor \frac{\projdim_S R}{p+1}\right\rfloor+1    
 \end{equation} 
 where the ``$+1$" can be omitted if $p+1$ divides $\projdim_S R$. 
 \end{theorem} 
 
 For more general results the reader can consult   \cite[Section 5]{ACI2}. 
 
 \begin{remark} 
 The problem of bounding the regularity of  Tor-modules has been studied in \cite{EHU}.  Let  $S=K[x_1,\dots,x_n]$,  and let   $M$ and $N$ be finitely generated graded $S$-modules. It is proved in \cite{EHU} that  if $\dim \Tor_1^S(M,N)\leq 1$, then for every $i$ one has
\begin{equation} 
\label{EHU}
\reg_S \Tor_i^S(M,N)\leq \reg_S(M)+ \reg_S(N)+i.
\end{equation} 
Unfortunately, the formula \ref{EHU} does not hold if we replace the polynomial ring $S$ with a Koszul ring $R$. 
For example with 
$$R=K[x,y]/(x^2+y^2), \quad    M=R/(\bar{x}) \quad \mbox{  and } N=R/(\bar{y})$$
 one has 
 $$\reg_R M=\reg_R N=0 \quad \mbox{  and } \quad  \Tor_1^R(M,N)=K(-2)$$ so that has regularity 
 $$\reg_R  \Tor_1^R(M,N)=2.$$
Nevertheless variations of \ref{EHU}  (e.g. compute the Tor over $R$ but regularity over $S$ or add a correction term on the left depending on $\reg_S(R)$) might hold in general.  
\end{remark}

The regularity bound in the Theorem \ref{mainACI2} is much weaker than the logarithmic one obtained by Dao, Huneke and Schweig  in \cite[4.8]{DHS}. They showed that an algebra $R$ with monomial quadratic relations and satisfying the property $N_p$ for some $p>1$   has a very low regularity compared with its embedding dimension.  Their result  asserts that for a given $p>1$ there exist  $f_p(x)\in \RR[x]$ and $\alpha_p\in \RR$ with $\alpha_p>1$ (which are explicitly given in the paper) such that 
 \begin{equation}
  \label{DHS}
  \reg_S R\leq \log_{\alpha_p}  f_p(n)
\end{equation}
 holds for every algebra $R$ with quadratic monomial relations such that $R$ has the property $N_p$ and has embedding dimension $n$.

 This type  of  bound cannot hold for Koszul algebras satisfying $N_p$  no matter what $f_p(x)\in \RR[x]$ and $\alpha_p \in (1,\infty)$ are.  
 To show this, it is enough to describe a family of algebras $\{R_{p,m}\}$ with $p,m\in \NN$ and $p>1$  such that:
 \begin{itemize} 
 \item[(1)] the algebra $R_{p,m}$ is Koszul and satisfies the $N_p$-property,
 \item[(2)] given $p$,  the embedding dimension is $R_{p,m}$ is a polynomial function of  $m$ and the regularity of $R_{p,m}$ is linear in $m$. 
 \end{itemize} 
For example, let  $R_{p,m}$  be the $p$-th Veronese subalgebra of a polynomial ring in $pm$ variables. Then $R_{p,m}$   is Koszul, it satisfies the $N_{p}$-property, it has regularity $(p-1)m$,  and has embedding dimension equal to $\binom{pm+p-1}{p}$, see \cite{BCR1}.

\begin{question}
Consider the coordinate ring of the Grassmannian $G(2,n)$.  It is defined by the Pfaffians of degree 2 (the $4\times4$ pfaffians) in a $n\times n$ skew-symmetric  generic matrix.  We know that it is Koszul, it satisfies the $N_2$ condition (by work of Kurano \cite{Ku}),  it has regularity $n-3$ and codimension $\binom{n-2}{2}$.  Hence for this family the codimension is quadratic in the regularity. 
Does there exist  a family like this (Koszul with the $N_2$ property)  such that the codimension is linear in the regularity?
 \end{question} 
 
 \begin{remark}
 \label{tgame}
 If we apply Theorem \ref{rimpo}  in the case $R=S=K[x_1,\dots,x_n]$ with $K$ of characteristic $0$ or large and $M$ any  graded module then we have a surjection 
$$\Tor_1^S(C_{i-1}, Z_j(M))\to  H_{i+j}(\mm_S,M) $$
because  $H_i  =0$. Here we have set for simplicity $Z_j(M)=Z_j(\mm_S, M)$.
Since 
$$\Tor_1^S(C_{i-1},Z_i(M))=H_i(\mm_S, Z_j(M))$$
 we obtain 
\begin{equation}
\label{surprise} 
\beta^S_{i,v}(Z_j(M))\geq \beta^S_{i+j,v}(M)
\end{equation} 
for all $i,j,v$ and every $M$. 
 \end{remark} 


  
  We conclude by presenting a short proof of the inequality $t_i^S(R)\leq 2i$ for Koszul algebras and a related question. 
  
  \begin{lemma} 
  \label{soeasy}
  Let $R$ be a Koszul algebra and $Z_i=Z_i(\mm_R)$ the cycles of the Koszul complex $K(\mm_R,R)$ associated with the maximal homogeneous ideal of $R$.  Then $\reg_R(Z_i)\leq 2i$ for every $i$. In particular, $t_i^S(R)\leq 2i$ for every $i$. 
  \end{lemma} 
  \begin{proof} 
  For $i>0$ we have short exact sequences: 
  $$0\to Z_{i}\to K_{i}\to B_{i-1} \to 0$$
  $$0\to B_{i-1}\to Z_{i-1}\to H_{i-1} \to 0.$$
  Hence one has: 
  $$\reg_R(Z_{i})=\reg_R(B_{i-1})+1$$
  and 
  $$\reg_R(B_{i-1})\leq \max\{ \reg_R(Z_{i-1}),  \reg_R(H_{i-1})+1\}.$$
  Hence 
  $$\reg_R(Z_{i})\leq  \max\{ \reg_R(Z_{i-1})+1,  \reg_R(H_{i-1})+2\}.$$
  Since $\mm_RH_{i-1}=0$ and $R$ is Koszul we have 
  $$\reg_R(H_{i-1})=t_0^R(H_{i-1})\leq t^R_0(Z_{i-1})\leq \reg_R(Z_{i-1}).$$
  It follows that 
$$\reg_R(Z_{i})\leq  \max\{ \reg_R(Z_{i-1})+1,  \reg_R(Z_{i-1})+2\}=\reg_R(Z_{i-1})+2.$$
Since $Z_0=R$ one has $\reg_R(Z_0)=0$ and it  follows by induction that 
$$\reg_R(Z_i)\leq 2i.$$
Since 
$$t_i^S(R)=t_0^R(H_i)\leq t_0^R(Z_i)\leq \reg_R(Z_i)\leq 2i$$
we may conclude that 
$$t_i^S(R)\leq 2i.$$
\end{proof} 

With the assumptions and notation of \ref{soeasy} one can ask: 
\begin{question} 
  \label{subadd2} Does the inequality 
   $$\reg_R(Z_{i+j})\leq \reg_R(Z_{i})+ \reg_R(Z_{j})$$ hold for every $i,j$?     
   \end{question} 
   
   For similar questions and results the reader can consult \cite{CM}.

 \section{Veronese rings and algebras associated with families of hyperspaces} 
In the third lecture we present two case studies: the Koszul properties of Veronese rings and of algebras associated with families of hyperspaces. 
The material we present is taken from \cite{ABH,BF,B1,BaM,Ca,CC,CHTV, CTV, CRV,ERT}. 

\subsection{Veronese rings}  
We will use the following results whose proofs can be found, for example, in the survey paper  \cite{CDR}. 

\begin{lemma}
\label{blabla}
Let $R$ be a standard graded $K$-algebra. 
Let $$\MM: \cdots \to M_i\to\cdots\to M_2\to M_1\to M_0\to 0$$ be a complex  of $R$-modules. Set $H_i=H_i(\MM)$. Then for every $i\geq 0$ one has
 $$t_i^R(H_0)\leq \max\{\alpha_i,\beta_i\}$$ where 
 $\alpha_i=\max\{ t_j^R(M_{i-j}) : j=0,\dots, i\}$
 and  $\beta_i=\max\{ t_j^R(H_{i-j-1}) : j=0,\dots,i-2\}$. 
 
\noindent Moreover one has  
$$\reg_R(H_0)\leq \max\{\alpha,\beta\}$$
 where $\alpha=\sup\{ \reg_R (M_j)-j : j\geq 0\}$ and $\beta=\sup\{ \reg_R (H_j)-(j+1) : j\geq 1\}$.
 \end{lemma}

\begin{theorem} 
\label{Giux1}
Let $A$ be a standard graded $K$-algebra,  $J\subset A$ an homogeneous ideal and $B=A/J$.  Then: 
\begin{itemize} 
\item[(1)] If  $\reg_A(B)\leq 1$ and $A$ is Koszul, then $B$ is Koszul. 
\item[(2)] If  $\reg_A(B)$ is finite  and $B$ is Koszul, then $A$ is Koszul. 
\end{itemize} 
\end{theorem}

 We apply now Theorem \ref{Giux1} to prove that the Veronese subrings of a Koszul algebra are Koszul.  
 
 Let  $R$ be a standard graded $K$-algebra. Let $c\in \NN$ and 
 $$R^{(c)}=\oplus_{j\in \ZZ}  R_{jc}$$ be  the $c$-th Veronese subalgebra of $R$.  Similarly, given a graded $R$-module $M$  one defines 
 
 $$M^{(c)}=\oplus_{j\in \ZZ}  M_{jc}.$$ 
 
 The formation of the $c$-th Veronese submodule can be seen as an exact functor from the category of  graded $R$-modules and maps of degree $0$   to  the category of graded $R^{(c)}$-modules and maps of degree $0$.  For $u=0,\dots, c-1$ consider the Veronese submodules 
 $$V_u=\oplus_{j\in \ZZ}  R_{jc+u}$$ of $R$.  Note that $V_u$ is an $R^{(c)}$-module generated in degree $0$ and that for $a\in \ZZ$ one has 
 $$R(-a)^{(c)}=V_u(-\lceil a/c \rceil )$$  where $u=0$ if $a\equiv 0$  mod$(c)$ and $u=c-r$ if  $a\equiv r$  mod$(c)$ and $0<r<c$. 
  
 \begin{theorem}  
 \label{cc1}
 Let $R$ be a Koszul algebra.  Then for every $c\in \NN$ one has: 
 \begin{itemize} 
 \item[(1)] $R^{(c)}$ is Koszul.
 \item[(2)]  $\reg_{R^{(c)}}(V_u)=0$ for every $u=0,\dots, c-1$, i.e. the Veronese submodules of $R$ have a linear resolution over $R^{(c)}$. 
 \end{itemize} 
 \end{theorem} 

 \begin{proof} Denote by  $A$ the ring $R^{(c)}$.  We  first prove assertion (2). Indeed we prove by induction on $i$ that $t_i^{A}(V_u)\leq i$ for every $i$. There is nothing to prove in the  case $i=0$. For $i>0$, observe that, since $R$ is Koszul,  one has $\reg_R \mm=1$ and by induction one has  that   $\reg_R \mm^u=u$. Now  let $M=\mm_R^u(u)$ so that  $\reg_R(M)=0$ and $M^{(c)}=V_u$. Consider the minimal free resolution $\FF$ of $M$ over $R$ and apply the functor $-^{(c)}$. We get a complex $\GG=\FF^{(c)}$ of $A$-modules such that $H_0(\GG)=V_u$,  $H_j(\GG)=0$ for $j>0$ and $G_j=F_j^{(c)}$ is a direct sum of copies of $R(-j)^{(c)}$.  Applying \ref{blabla} and the inductive assumption we get $t^A_i(V_u)\leq i$ as required. 
 
 To prove that $A$ is Koszul we consider the minimal free resolution $\FF$ of $K$ over $R$  and apply $-^{(c)}$. We get a complex $\GG=\FF^{(c)}$ of $A$-modules such that $H_0(\GG)=K$, $H_j(\GG)=0$ for $j>0$ and $G_j=F_j^{(c)}$ is a direct sum of copies of  $V_u(-\lceil j/c \rceil )$. Hence $\reg_A ( \GG_j)=\lceil j/c \rceil$ and applying \ref{blabla} we obtain 
 $$\reg_A(K)\leq \sup \{ \lceil j/c \rceil-j : j\geq 0\}=0.$$ 
  \end{proof}
 
 We also have: 
 
 \begin{theorem} 
 \label{cc2}
 Let $R$ be a standard graded $K$-algebra.  Then: 
 \begin{itemize}
 \item[(1)]   The Veronese subalgebra $R^{(c)}$ is Koszul for $c\gg 0.$ 
 \item[(2)] If $R=S/I$ with $S=K[x_1,\dots,x_n]$, then $R^{(c)}$ is Koszul for every $c$ such that 
 $$c\geq \max\{ t_i^S(R)/(1+i)  : i\geq 0\}.$$ 
 \end{itemize} 
  \end{theorem} 
  \begin{proof} Let $\FF$ be the minimal free resolution of $R$ as a $S$-module.  Set $B=S^{(c)}$ and note that $B$ is Koszul because of \ref{cc1}. Then $\GG=\FF^{(c)}$ is a complex of $B$-modules such that $H_0(\GG)=R^{(c)}$,  $H_j(\GG)=0$ for $j>0$. Furthermore $G_i=F_i^{(c)}$ is a direct sum of shifted copies of the Veronese submodules $V_u$. Using \ref{cc1} we get  the bound $\reg_B(G_i)\leq \lceil t_i^A(R)/c \rceil$. Applying \ref{blabla} we get 
  $$\reg_B (R^{(c)})\leq \max \{ \lceil t_i^S(R)/c \rceil -i : i\geq 0\}.$$
  Hence for $c\geq \max\{ t_i^S(R)/(1+i)  : i\geq 0\}$ one has $\reg_B (R^{(c)})\leq 1$ and we conclude from \ref{Giux1} that $R^{(c)}$ is Koszul. 
  \end{proof}

  \begin{remark} 
 (1) In \cite[2]{ERT} it is proved that if $c\geq (\reg_S(R)+1)/2$, then $R^{(c)}$ is even G-quadratic. See \cite{Sh} for other interesting results in this direction. 

(2) Backelin proved in \cite{B1} that $R^{(c)}$ is Koszul if $c\geq \Rate(R)$.  Here $\Rate(R)$ is defined as 
$$\sup_{i>0} \{  (t_{i+1}^R(K)-1)/i\} $$  and it is finite. It  measures the deviation from the Koszul property in the sense that $\Rate(R)\geq 1$ with equality if and only if $R$ is Koszul. 
  \end{remark}

\subsection{Strongly Koszul algebras}

A powerful  tool for proving that an algebra is Koszul is a  typical ``divide and conquer" strategy that can be formulated in the following way, see \cite{CTV} and \cite{CRV}: 

\begin{definition} A Koszul filtration of a $K$-algebra $R$ is a set $\F$ of ideals of $R$ such that: 
\begin{itemize} 
\item[(1)]  Every ideal $I\in \F$ is generated by elements of degree $1$.
\item[(2)] The zero ideal $0$ and the maximal ideal $\mm_R$ are in $\F$. 
\item[(3)] For every $I\in \F$, $I\neq 0$, there exists $J\in \F$ such that $J\subset I$,  $I/J$ is cyclic and $\Ann(I/J)=J:I\in \F$. 
\end{itemize}
\end{definition} 
 
One easily proves: 

\begin{lemma} 
\label{filtflag} Let $\F$ be a Koszul filtration of a standard graded $K$-algebra $R$. Then one has: 
\begin{itemize} 
\item[(1)]  $\reg_R (R/I)=0$ and $R/I$ is Koszul for every $I\in \F$.   
\item[(2)]  $R$ is Koszul. 
\end{itemize}
\end{lemma}

\begin{example} 
\label{non-LG with filtration}
Let 
  $$R=K[a,b,c,d]/(ac, ad, ab - bd, a^2 + bc, b^2).$$
 We have seen in  Example \ref{KosNotLG} that $R$ is not LG-quadratic.  We show now that $R$ is Koszul by constructing a Koszul filtration. 
 Indeed, there is a Koszul filtration based on the given system of coordinates, i.e. a Koszul filtration whose ideals are generated by residue classes  of variables. Here it is: 
 $$\F=\{ (a, b, c, d), (a, c, d), (c, d), (a, c), (c), (a), 0\}.$$
 To check that it is a Koszul filtration we observe that in $R$ one has: 
$$
\begin{array}{rl}
(a, c, d):& \!\!\!\! \!\! (a, b, c, d)=(a, b, c, d) \\
(c, d):&  \!\!\!\! \!\!(a, c, d) =(a, b, c, d) \\
(c):& \!\!\!\! \!\!(c, d)= (a, c)\\
(c):& \!\!\!\! \!\!(a, c)=(a, c, d) \\
0:& \!\!\!\! \!\!(a)=(c, d) \\
0:& \!\!\!\! \!\!(c)=(a)
\end{array} 
$$
\end{example}

 The following notion is  very natural for algebras with a canonical coordinate system (e.g. in  the toric case). 

\begin{definition} An algebra $R$ is strongly Koszul if there exists a basis $X$ of $R_1$ such that for every $Y\subset X$ and for every $x\in X\setminus Y$ there exists $Z\subseteq  X$ such that $(Y):x=(Z)$. 
\end{definition} 

This  definition of strongly Koszul is taken from \cite{CDR} and is slightly different from the one given in \cite{HHR}. In \cite{HHR} it is assumed that the  basis $X$ of $R_1$ is totally ordered  and in the definition one adds the requirement that $x$ is larger than every element in $Y$. 
  
\begin{remark} If $R$ is strongly Koszul with respect to a basis $X$ of $R_1$ then the set $\{ (Y) : Y\subseteq X\}$ is obviously a Koszul filtration. \end{remark} 

We have: 
\begin{theorem}
\label{sKos}
Let $R=S/I$ with $S=K[x_1,\dots,x_n]$ and $I\subset S$ be an ideal generated by monomials of degrees $\leq d$. 
Then $R^{(c)}$ is strongly Koszul for every $c\geq d-1$.
\end{theorem}

The proof of Theorem \ref{sKos} is based on the fact that the Veronese ring $R^{(c)}$ is a direct summand of $R$ and that  computing the colon ideal of monomial  ideals in a polynomial ring  is a combinatorial operation. 
 Let us single out an  interesting special case: 

\begin{theorem}
\label{monKos}
Let $S=K[x_1,\dots,x_n]$ and let $I\subset S$ be an ideal generated by monomials of degree $2$. Then $S/I$ is strongly Koszul. 
 \end{theorem}

The results  presented for Veronese rings and Veronese modules have their analogous in the multigraded setting, see \cite{CHTV}.  We discuss below    the bigraded case.

Let $$S=K[x_1,\dots,x_n,y_1,\dots,y_m]$$  with $\ZZ^2$-grading induced by the assignment $\deg(x_i)=(1,0)$ and $\deg(y_i)=(0,1)$. For every $c=(c_1,c_2)$ we look at the diagonal subalgebra $$S_\Delta=\oplus_{a\in \Delta} S_{a}$$ where $$\Delta=\{ic : i\in \ZZ\}.$$ 
The algebra $S_\Delta$ is nothing but the Segre product of the $c_1$-th Veronese ring of $K[x_1,\dots,x_n]$ and the $c_2$-th Veronese ring of $K[y_1,\dots,y_m]$. 
For a $\ZZ^2$-graded standard $K$-algebra $R=S/I$ with  $I\subset S$  a bigraded ideal we may consider the associate diagonal algebra 
$$R_\Delta=\oplus_{a\in \Delta} R_{a}$$ and similarly for modules. One has: 

\begin{theorem}
\label{Segremod}
\begin{itemize}
\item[(1)]  For every $(a,b)\in \ZZ^2$ the  $S_\Delta$-submodule $S(-a,-b)_\Delta$ of $S$ has a linear resolution. 
\item[(2)] For every $\ZZ^2$-standard graded algebra $R$ one has that $R_\Delta$ is Koszul for ``large" $\Delta$ (i.e. $c_1\gg 0$ and $c_2\gg 0)$. One can give explicit  bounds in terms of the bigraded Betti numbers of $R$ as a $S$-module. 
\end{itemize} 
 \end{theorem}

 Let  $I$ be an homogeneous ideal of $S=K[x_1,\dots,x_n]$ generated by elements $f_1,\dots,f_r$ of degree $d$. The Rees ring 
 $$\Rees(I)=\bigoplus_{i\in \NN}  I^i=S[f_1t,\dots,f_rt]\subset S[t]$$
  is  a bigraded  $K$-algebra. Its component of degree $(i,j)$ is  
  $$\Rees(I)_{(i,j)}=(I^j)_{jd+i}$$
  It is easy to check that $\Rees(I)$ is a  standard bigraded algebra. It  can be seen as a quotient  ring of 
  $S[y]=S[y_1,\dots,y_r]$ bigraded by $\deg(x_i)=(1,0)$ and $\deg(y_i)=(0,1)$. 
  Then we may apply Theorem \ref{Segremod} and we get that
 
  \begin{corollary}
  \label{seserompe} 
There exist integers $c_0$ and $e_0$ (depending on $I$) such that  for every  $c\geq c_0$ and $e\geq e_0$ the $K$-subalgebra of $S$ generated by the vector space $(I^{e})_{ed+c}$ is Koszul.
  \end{corollary}
  
  If one has information or bounds on the bigraded resolution of $\Rees(I)$ as a $S[y]$-module then  Corollary \ref{seserompe}  can be formulated more precisely. One of the few families of ideals $I$ for which the resolution of $\Rees(I)$ is known are the complete intersections. If $f_1,\dots,f_r$ form a regular sequence then 
  $$\Rees(I)=S[y]/I_2
  \left(
  \begin{array}{cccc}
  y_1 & y_2& \dots & y_r\\
  f_1 & f_2& \dots & f_r
  \end{array}
 \right)
 $$ 
 and $\Rees(I)$ is resolved by the Eagon-Northcott complex.  Then applying the principle described above to this specific case one has: 
 
 \begin{theorem}
  \label{seserompe1} 
 Let $f_1,\dots,f_r$  be  a regular sequence of elements of degree $d$ in $S=K[x_1,\dots,x_n]$ and $I=(f_1,\dots,f_r)$. For $c,e\in \NN$ set $A=K[(I^e)_{ed+c}]$. Then: 
 \begin{itemize}
 \item[(1)]   If $c\geq d/2$ then $A$ is quadratic. 
 \item[(2)]   if $c\geq d(r-1)/r$ then $A$ is Koszul. 
\end{itemize} 
 \end{theorem}

 See  \cite{CHTV} for  details of the proofs of \ref{seserompe1}. 
  
  \begin{example}
  \label{diagTru} 
  With $r=n$, $d=2$ and $f_i=x_i^2$ for every $i=1,\dots,n$  we have that the toric algebra
  $$K[x^a : a\in \NN^n, |a|=2+c \mbox{ and } \max(a)\geq 2]$$
  is quadratic for every $c$ and Koszul for $c\geq 2(n-1)/n$. 
  \end{example} 
  
Given integers $n,d,s$ we set 
   $$\PV(n,d,s)=K[x^a : a\in \NN^n, |a|=d \mbox{ and } \#\{ i  : a_i>0\}\leq s].$$
   This is called the pinched Veronese generated by the monomials in $n$ variables, of total degree $d$ and supported on at most $s$ variables. 
  
   \begin{question}   For which values of $n,d,s$ is the algebra $\PV(n,d,s)$ quadratic or Koszul?  Not all of them are quadratic, for instance $\PV(4,5,2)$ is defined,  according to CoCoA \cite{CoCoA},  by $168$ quadrics and $12$ cubics. \end{question} 

The algebra of Example \ref{diagTru}  for $c=n-2$ coincides with  the pinched Veronese $\PV(n,n,n-1)$. 
Hence $\PV(n,n,n-1)$ is quadratic for every $n$ and Koszul for $n>3$. For $n=3$  we have that 
$$\PV(3,3,2)=K[x^3,x^2y,x^2z,xy^2,xz^2, y^3, y^2z, yz^2,z^3]$$
is quadratic. The argument above does not answer the question whether $\PV(3,3,2)$ is a Koszul algebra. This turns out to be a difficult question on its own.  In  \cite{Ca} and \cite{CC} it is proved that: 

 \begin{theorem}
  \label{pV332} 
  The pinched Veronese $\PV(3,3,2)$ is Koszul. The same hold for the generic projection of the Veronese surface of $\PP^9$ to $\PP^8$. 
  \end{theorem}
  It is not clear whether  $\PV(3,3,2)$  is G-quadratic. 
  The Koszul property of a toric ring is equivalent to the Cohen-Macaulay property of intervals of the underlying poset, see \cite[2.2]{PRS}. 
 Recently Tancer has shown that the intervals of the  poset associated with  $\PV(n,n,n-1)$ are shellable for $n>3$, see \cite{Ta}. It is not clear whether the same is true for $n=3$. 
  
  \subsection{Koszul algebras associated with  hyperspace configurations} 
  Another interesting family of Koszul algebras with relations to combinatorics arises in the following way. 
  Let $V=V_1, \dots, V_m$ be a collection of  subspaces of the space of linear forms in the polynomial ring $K[x_1,\dots,x_n]$.
  Denote by $A(V)$ the $K$-subalgebra of $K[x_1,\dots,x_n]$ generated by the elements of the product $V_1\cdots V_m$.  
  We have: 
   \begin{theorem}
  \label{palg} 
The algebra  $A(V)$ is Koszul.   
  \end{theorem}
 We outline the proof of Theorem \ref{palg}. Denote by $R$ the polynomial ring $K[x_1,\dots,x_n]$ and set  $d_i=\dim V_i$.  Consider auxiliary variables $y_1,\dots, y_m$ and  the Segre product: 

$$S=K[y_ix_j :  i=1,\dots,m, j=1,\dots,n]$$
of $K[y_1,\dots,y_m]$ with $R$. Set  
$$B(V)=K[y_1V_1,\dots,y_mV_m].$$

and
 
$$T=K[t_{ij} :  i=1,\dots,m, j=1,\dots,n].$$

Note that  $B(V)$ is a $K$-subalgebra of  $S$. We give degree
$e_i\in {\bf Z}^m$ to $y_ix_j$ and to $t_{ij}$ so that  
$S$, $T$  and $B(V)$ are ${\bf Z}^m$-graded.  
Let $$\Delta=\{ (a,a,\dots,a)\in \ZZ^m : a\in \ZZ\}.$$
By construction, the diagonal algebra $$B(V)_\Delta=\bigoplus_{b\in \ZZ^m} B(V)_b$$ coincides with $K[V_1\cdots V_my_1\cdots y_m]$ and  hence 
$$B(V)_\Delta=A(V).$$
 
 For $i=1,\dots, m$ let $\{ f_{ij} :  j=1,\dots,d_i\}$  be a basis of $V_i$ and complete it to a basis of $R_1$ with elements  $\{ f_{ij} : j=d_i+1,\dots,n\}$ (no matter how).
Set $T(V)=K[t_{ij} : 1\leq i\leq m, 1\leq j\leq d_i]$.  We have presentations: 
$$\begin{array}{lll} 
\phi:  T     \to S &      \mbox{  with }  t_{ij}\to y_if_{ij}  &  \mbox{ for all } i,j \\  \\ 
\phi': T(V)\to B(V) &      \mbox{  with }  t_{ij}\to y_if_{ij}  & \mbox{ for all } i \mbox{ and } 1\leq j\leq d_i
\end{array}
$$
We have: 
\begin{lemma}\label{lemon1}
 Suppose that $\ker \phi'$ has a Gr\"obner basis (with respect to some term order $>$) of elements of degrees bounded above by $(1,1,\dots,1)\in \ZZ^m$. Then $A(V)$ is Koszul. 
\end{lemma} 
\begin{proof} Set $I=\ker \phi'$. Applying $\Delta$ to the presentation $B(V)=T(V)/I$ we obtain $A(V)=T(V)_\Delta/I_\Delta$. Now $T(V)_\Delta$ is a multiple Segre product and hence it is strongly Koszul (the argument is similar to the one for the Veronese case) and by assumption $\ini_>(I_\Delta)$ is generated by a subset of the semigroup generators of $T(V)_\Delta$. But then $T(V)/\ini_>(I_\Delta)$ is Koszul because of the strongly Koszul property of $T(V)_\Delta$. Hence $T(V)_\Delta/I_\Delta$ is Koszul by Gr\"obner deformation. 
\end{proof} 

Since, by construction, $\ker \phi'=\ker \phi \cap T(V)$, a Gr\"obner basis of $\ker \phi'$ can be obtained from a lexicographic Gr\"obner basis of $\ker \phi$ by elimination.  Therefore, combining this point of view  with \ref{lemon1} we have that Theorem \ref{palg} is a corollary of:
\begin{lemma} 
\label{lemon2} 
The ideal $\ker \phi$ has a universal Gr\"obner basis whose elements  of degrees bounded above by $(1,1,\dots,1)\in \ZZ^m$.  
\end{lemma} 
Observe that $\phi$ is a presentation of the Segre product $S$ but with respect to a non-necessarily monomial basis. Hence $\ker \phi$
 is obtained from the ideal $I_2(t)$ of the $2$-minors of the $m\times n$ matrix 
 $$t=(t_{ij})$$ 
 by a change of coordinates preserving the $\ZZ^m$-graded structure. 
 Since the Hilbert function does not change under taking initial ideals, it is enough to prove the following (very strong) assertion: 
 
 \begin{lemma} 
\label{lemon3} 
Every ideal of $T$ that has the $\ZZ^m$-graded Hilbert function  of $I_2(t)$ is generated in degrees  bounded above by $(1,1,\dots,1)\in \ZZ^m$.  
\end{lemma} 

Lemma \ref{lemon3} has been proved by Cartwright and Sturmfels \cite{CS} using multigraded generic initial ideals and a result proved in \cite{C}. This approach has been generalized in \cite{CDG} to identify universal Gr\"obner bases of ideals of maximal minors of matrices of linear forms hence generalizing the classical result of Bernstein, Sturmfels and Zelevinsky \cite{BZ, SZ}, see also \cite{K}.  In details,  the group    $\GL_n(K)^m$ acts as the group of $\ZZ^m$-graded $K$-algebra automorphisms on $T$ by linear substitution (row by row). An ideal $I\subset T$ is Borel-fixed if it is invariant under the action of the Borel subgroup $B_n(K)^m$ of $\GL_n(K)^m$. Here $B_n(K)$ is the group of upper triangular matrices. In \cite{CDG} is has been proved that 

 \begin{lemma} 
\label{lemon4} 
If $J\subset T$ is Borel-fixed and radical then every ideal $I$ with the $\ZZ^m$-graded Hilbert function  of $J$ is generated in degrees bounded above by $(1,1,\dots,1)\in \ZZ^m$. 
\end{lemma} 

Summing up, to conclude the proof of \ref{palg} it is enough to prove that: 

\begin{lemma} 
\label{lemon5} 
The ideal $I_2(t)$ has the $\ZZ^m$-graded Hilbert function of the radical and Borel-fixed  ideal $J$ generated by the monomials  $t_{i_1j_1}\cdots t_{i_kj_k}$ satisfying the following conditions: 
$$\begin{array}{l} 1\leq i_1<\dots<i_k\leq m, \\  1\leq j_1,\dots, j_k\leq n, \\   j_1+\dots+j_k\geq n+k. \end{array}$$
 \end{lemma} 
 This is done in \cite{C} by proving that $J$ is indeed the multigraded generic initial  ideal $\gin(I_2(t))$ of $I_2(t)$. 
 The inclusion $J\subseteq \gin(I_2(t))$ is a consequence of the following Lemma \ref{lemon6}.  The other inclusion is proved by checking that $J$ is pure with codimension and degree equal to that of $I_2(t)$. 
 
 \begin{lemma} 
\label{lemon6} 
Let  $V_1,\dots,V_m$  be  subspaces of the vector space of the linear forms $R_1$.   If   $\sum_{i=1}^m  \dim V_i \geq n+m$  then $\dim
\prod_{i=1}^m V_i < \prod_{i=1}^m  \dim V_i$, i.e.   there is a non-trivial linear relation  among the  generators of the
product  $\prod_{i=1}^m V_i$ obtained by multiplying $K$-bases of the $V_i$.
 \end{lemma} 
 
 One can also prove directly that $T/I_2(t)$ and $T/J$ have the same $\ZZ^m$-graded Hilbert function in the following way. 
 For every   $a=(a_1,\dots,a_m)\in {\NN}^m$ we  
show that the vector space dimensions of $T/I_2(t)$ and of $T/J$ in multidegree $a$ are equal. By
induction on $m$ we may assume that $a$ has full support, i.e. $a_i>0$ for all $i$.  Since  $T/I_2(t)$ is the Segre product of 
 $K[y_1,\dots,y_m]$ and $K[x_1,\dots,x_n]$, a $K$-basis of $T/I_2(t)$ in degree $a$ is given by the monomials of the form
$\prod y_i^{a_i}p$ where $p$ is a monomial in the $x$'s of degree $\sum a_i$.  It
follows that  the dimension of
$T/I_2(t)$   in degree
$a$ equals: 
$$\binom{n-1+a_1+a_2+\dots+a_m  }{  n-1}.$$ 
Given a monomial  $p$    in the $t_{ij}$'s of multidegree $a$ we  set
$M_i(p)=\max\{ j : t_{ij}| p\}$.   It is easy to see that: 
$$p\in J \iff  \sum_{i=1}^m  M_i(p)\geq n+m.$$
For a fixed $c=(c_1,\dots,c_m)\in \{1,\dots, n\}^m$ the cardinality of the set
of the monomials $p$ in the $t_{ij}$'s with degree $a$ and  $M_i(p)=c_i$ is given by 
$$\prod_{i=1}^m   \binom{c_i-1+a_i-1 }{  c_i-1}.$$
Therefore the dimension of $T/J$ in multidegree $a$ is given by: 
$$\sum_c   \prod_{i=1}^m   \binom{c_i-1+a_i-1 }{  c_i-1}$$
where the sum is extended to all the $c=(c_1,\dots,c_m)\in \{1,\dots, n\}^m$ with   
$c_1+\dots+c_m<n+m$.  Replacing  $n-1$ with $n$ and $c_i-1$ with $c_i$,  we have to
prove the following identity: 
\begin{equation}
\label{fk1} 
\binom{n+a_1+a_2+\dots+a_m  }{  n}=\sum_c   
\prod_{i=1}^m   \binom{c_i+a_i-1 }{  c_i}
\end{equation}
where the sum is extended to all the $c=(c_1,\dots,c_m)\in {\NN}^m$ with   
$c_1+\dots+c_m\leq n$. 
The equality \ref{fk1}  is a specialization ($v=m+1$, $b_i=a_i$  for $i=1,\dots, m$ and $b_{m+1}=1$) of  the following identity: 
 
\begin{equation}
\label{fk2} 
\binom{n+b_1+b_2+\dots+b_v-1  }{  n}=\sum_c   
\prod_{i=1}^v   \binom{c_i+b_i-1 }{  c_i}
\end{equation}

where the sum is extended to all the $c=(c_1,\dots,c_v)\in {\NN}^v$ with   
$c_1+\dots+c_v= n$. 

Now the identity \ref{fk2} is easy:  both the  left and right side of it  count the number of
monomials of total degree $n$ in a set of variables which is a disjoint union of
subsets of cardinality $b_1,b_2,\dots,b_v$. 

\bigskip

\noindent{\bf Thanks}: We thank Giulio Caviglia, Alessio D'Ali, Emanuela De Negri and Dang Hop Nguyen for their valuable comments and suggestions upon reading preliminary versions of the present notes and Christian Krattenthaler for suggesting the proof of formula \ref{fk1}.

\end{document}